\documentclass[a4paper,10pt,reqno]{amsart}

\usepackage{
	amsfonts,
	amsmath,
	amssymb,
	amsthm,
	tikz,
	hyperref,
	graphicx,
	caption,
	subcaption,
	float,
	todonotes,
        bm,
        a4wide,
        enumerate,
        microtype,
}

\usepackage{color}

\newcommand{\CC}{\mathrm{\mathbb{C}}}

\newcommand{\NN}{\mathrm{\mathbb{N}}}

\newcommand{\RR}{\mathrm{\mathbb{R}}}

\newcommand{\diag}{\mathrm{diag}}

\newcommand{\del}{\partial}

\newcommand{\SU}{\mathrm{SU}}
\newcommand{\U}{\mathrm{U}}
\newcommand{\SO}{\mathrm{SO}}
\newcommand{\Sp}{\mathrm{Sp}}

\newcommand{\bbN}{\mathbb{N}}

\newcommand{\rFs}[5]{\,_{#1}F_{#2} \left(\genfrac{.}{.}{0pt}{}{#3}{#4}\ ;#5 \right)}

\newtheorem{theorem}{Theorem}[section]
\newtheorem{proposition}[theorem]{Proposition}
\newtheorem{lemma}[theorem]{Lemma}
\newtheorem{definition}[theorem]{Definition}
\newtheorem{corollary}[theorem]{Corollary}
\theoremstyle{definition}

\newtheorem{remark}[theorem]{Remark}

\usetikzlibrary{matrix,arrows,patterns}

\numberwithin{equation}{section}

\begin{document}

\title{ Deformation of matrix-valued orthogonal polynomials related to Gelfand pairs} 

\date{\today} 

\author[van Pruijssen]{M. van Pruijssen}
\author[Roman]{P. Rom\'an}

\address[van Pruijssen]{Institut f\"ur Mathematik, Universit\"at Paderborn, Warburgerstrasse 100, 33098 Paderborn, Germany}
\address[Roman]{CIEM, FaMAF, Universidad Nacional de C\'ordoba, Medina Allende s/n Ciudad Universitaria, C\'ordoba, Argentina}

\begin{abstract}
In this paper we present a method to obtain deformations of families of matrix-valued orthogonal polynomials that are associated to the representation theory of compact Gelfand pairs. These polynomials have the Sturm-Liouville property in the sense that they are simultaneous eigenfunctions of a symmetric second order differential operator and we deform this operator accordingly so that the deformed families also have the Sturm-Liouville property. Our strategy is to deform the system of spherical functions that is related to the matrix-valued orthogonal polynomials and then check that the polynomial structure is respected by the deformation. Crucial in these considerations is the full spherical function $\Psi_{0}$, which relates the spherical functions to the polynomials. We prove an explicit formula for $\Psi_{0}$ in terms of Krawtchouk polynomials for the Gelfand pair $(\SU(2)\times\SU(2),\diag(\SU(2)))$. For the matrix-valued orthogonal polynomials associated to this pair, a deformation was already available by different methods and we show that our method gives same results using explicit knowledge of $\Psi_{0}$.

Furthermore we apply our method to some of the examples of size $2\times2$ for more general Gelfand pairs. We prove that the families related to the groups $\SU(n)$ are deformations of one another. On the other hand, the families associated to the symplectic groups $\Sp(n)$ give rise to a new family with an extra free parameter. 

\end{abstract}


\maketitle

\section{Introduction and statement of results}

It is  well known that the classical orthogonal polynomials are characterized by
the property that their derivatives are also orthogonal polynomials, see
e.g.~\cite{Hahn}. One can exploit this characterization, for instance, to
construct the Gegenbauer polynomials $C^{(\nu)}_n(y)$, for integer values of
$\nu$,  by repeated differentiation of the Chebyshev polynomials
$U_n(y)=C^{(1)}_n(y)$. In this way, the orthogonality relations, differential
equations and other properties of the Gegenbauer polynomials can be obtained from
those of the Chebyshev polynomials. In this paper we deal with the analogous
construction for matrix-valued orthogonal polynomials.

For $N\in \mathbb{N}$, let $W:(a,b)\to \mathbb{C}^{N\times N}$ be a positive
definite smooth weight matrix with finite moments. We consider the
sesqui-linear pairing defined for a pair of matrix polynomials $P,Q\in 
\mathbb{C}^{N\times N}[y]$ by the matrix \begin{equation} \label{eq:MV_inner_P}
\langle P, Q \rangle_W = \int_I \, P(y)W(y)Q(y)^* \,  dy, \end{equation} where
$Q(x)^*$ denotes the conjugate transpose of $Q(x)$. We say that $(P_d)_{d}$,
$P_d \in  \mathbb{C}^{N\times N}[x], d\in\bbN_{0}$, is a sequence of matrix-valued orthogonal
polynomials (MVOPs from now on) with respect to \eqref{eq:MV_inner_P} if
\begin{enumerate} \item $\langle P_d, P_{d'} \rangle_W = 0$ if $d\neq d'$, \item
$\deg P_d=d$, for all $d\in \mathbb{N}_0$, \item the leading coefficient of
$P_d$ is invertible for all $d\in\mathbb{N}$. \end{enumerate} A sequence of
monic MVOPs can be obtained by applying the Gram--Schmidt process on the ordered
basis $(I, yI,y^2I,\ldots)$, where $I$ denotes the $N\times N$ identity matrix.

We say that the differential operator $D:\mathbb{C}^{N\times N}[y]\to
\mathbb{C}^{N\times N}[y]$ is symmetric with respect to the matrix weight $W$ if
it satisfies $\langle PD,Q\rangle = \langle P,QD\rangle$ for all matrix
polynomials $P,Q$. Any sequence $(P_d)_d$ of MVOPs with respect to $W$ is a
basis of the space of matrix polynomials, so that the symmetry condition is
equivalent to 
\begin{equation} 
\label{eq:Symmetry_D_Pd} \langle P_dD,P_{d'}
\rangle = \langle P_d,P_{d'}D \rangle,\quad\mbox{for all $d,d'\in\mathbb{N}_0$.}
\end{equation} 
A pair $(W,D)$ consisting of a matrix weight $W$ together with a
matrix-valued differential operator $$D=\partial_y^2 \, F_2 + \partial_y \, F_1
+ F_0,$$ where $F_i$ is a polynomial of degree at most $i$, which is symmetric
with respect to $\langle\cdot,\cdot\rangle$ is called a matrix-valued classical
pair (MVCP from now on). Given a MVCP $(W,D)$, any sequence of MVOPs with
respect to $W$ is a family of simultaneous eigenfunctions of $D$, see
\cite[Proposition 2.10]{GrunT}.

\begin{definition} 
A deformation of a MVCP $(W,D)$ is a family
$(W^{(\kappa)},D^{(\kappa)})_{\kappa\in K}$ of MVCPs, where $K\subset\mathbb{R}$
is an open interval, so that $(W,D)=(W^{(\kappa_{0})},D^{(\kappa_{0})})$ for
some $\kappa_{0}\in K$.

We say that a deformation $(W^{(\kappa)}, D^{(\kappa)})_{\kappa\in K}$ allows
for the shift $\partial_{y}$, if for a family of MVOPs $(P_{d}^{(\kappa)})_{d}$
of $W^{(\kappa)}$, we have that $(\partial_{y}
P^{(\kappa)}_{d})_{d}$ is a family of MVOPs for
$W^{(\kappa+1)}$. 
\end{definition}

The goal of this paper is to present a method to find deformations that allow
for the shift $\partial_{y}$, of MVCPs $(W,D)$ that are associated to compact
Gelfand pairs of rank one \cite{HvP, vP}. We also present a Rodrigues type
formula for these cases. We apply our method to some examples of size $2\times
2$ taken from \cite{vPR}, where we observe that some of the families actually
fit in a single deformation, that moreover, allows for the shift $\partial_{y}$.

An earlier result in this direction is the case studied in \cite{KdlRR}, where
sequences of matrix-valued Chebyshev polynomials  \cite{KvPR,KvPR2} are deformed
into matrix-analogues of Gegenbauer polynomials. However, those results are
based on a decomposition for the weight matrix that is quite particular for the
Gelfand pair $(\SU(2)\times\SU(2),\diag(\SU(2)))$. If we apply our method to
this case, we obtain the same deformations.\\

\textbf{MVOPs associated to representation theory.} 
The theory of MVOPs dates
back to the work of M.G. Krein in the 1940s. Since then, the theory has been
developed and connected to different fields such as scattering theory  and
spectral analysis, see for instance \cite{Berez,Geronimo,
Groenevelt--Ismail--Koelink, Groenevelt--Koelink}. In \cite{Duran97} A. Dur\'an
raises the question of whether it is possible to construct a sequence of MVOPs
together with a matrix valued differential operator which has the MVOPs as
simultaneous eigenfunctions. The first results in this direction are  given in
\cite{GrunPT} from the study of matrix-valued spherical functions on
$\SU(3)/\mathrm{U}(2)$. The link between matrix-valued spherical functions and
MVOPs has been developed in several papers following \cite{Koornwinder1985,
GrunPT}, see for instance \cite{KvPR,KvPR2,TZ}. This led to a uniform
construction of MVOPs for compact Gelfand pairs $(G,K)$ of rank one \cite{HvP,
vP, vPR}. The families of MVOPs obtained from representation theory have many
interesting properties and, in some cases, can be described in great detail.
Certain families of MVOP related to the pair $(\SU(n+1),\U(n))$ have been
exploited to derive stochastic models, see \cite{GdI2,GPT3,dI1}. The family
described in Section \ref{sec:MVchebyshev}  leads to models of continuous-time
bivariate Markov processes which are analyzed in detail in \cite{dlIR}.

In this paper we deal with the families of MVOPs obtained in \cite{HvP, vP,
vPR}. We develop a method to deform these families that relies in a specific
decomposition of the weight matrix and the differential operator, see
\eqref{eq:structure_W_group_level} and \eqref{eq:Omega_1}. The same
decomposition arises in families of MVOPs constructed independently from group
theory as well, see for instance \cite[(4.8)]{DuraG}. Our
method could also be carried out in these cases, to investigate the existence of deformations. The construction of MVOPs that we consider in this paper applies to
a triple $(G,K,\mu)$, where $(G,K)$ is a compact Gelfand pair of rank one and
$\mu$ is a suitable representation of $K$, see \cite[Table 1]{vPR}. The output
is a family of $\mathbb{C}^{N_\mu\times N_\mu}$-valued functions
$\{\Psi^{\mu}_d:d\in \mathbb{N}_0\}$, defined on the interval $[0,1]$, together
with a matrix-valued differential operator $\Omega^\mu$ for which the
$\Psi^{\mu}_d$ are eigenfunctions. The function $\Psi^{\mu}_d$ is called the
{\it  full spherical function} of type $\mu$ and degree $d$. It follows from
\cite{HvP,vP} or \cite[Theorem 2.7]{vPR} that 
$$\Psi^{\mu}_d (y) = P^{\mu}_d(y)\Psi^{\mu}_0(y),\qquad y\in[0,1],$$ 
for all $d\in \NN$, where $P^{\mu}_d$ is a polynomial of degree $d$.

As a consequence of the Schur orthogonality relations, the full spherical
functions $(\Psi^{\mu}_d)_{d}$ are pairwise orthogonal. More precisely, in
\cite{HvP, vP} it is proven that 
\begin{equation} 
\label{eq:orthogonality_Psi}
\int_0^1 \Psi^{\mu}_d(y) T^{\mu} (\Psi^{\mu}_{d'}(y))^* \, (1-y)^\alpha y^\beta
\, dy = 0,\qquad d\neq d', 
\end{equation} 
where $ (1-y)^\alpha y^\beta$ is the
ordinary Jacobi weight on the interval $[0,1]$ that is associated to the Riemann symmetric space $G/K$
and $T^{\mu}$ is a constant diagonal matrix. The
spherical functions are eigenfunctions of the Casimir operator of the group $G$.
This implies that the full spherical function $\Psi^{\mu}_d$ is an eigenfunction
of a single variable differential operator $\Omega^{\mu}$, the radial part of
the Casimir operator. In other words there is an operator 
\begin{equation}
\label{eq:Omega_1} 
\Omega^{\mu}=y(1-y)\partial_y^2 + a(y)\partial_y+F^{\mu}(y),
\end{equation} such that \begin{equation} \label{eq:Psi_d_eig_Omega}
\Psi^{\mu}_d(y) \, \Omega^{\mu} = y(1-y)(\Psi^{\mu}_d)''(y) + a(y)
(\Psi^{\mu}_d)'(y) + \Psi^{\mu}_d(y)\, F^{\mu}(y)= \Lambda^{\mu}_d \,
\Psi^{\mu}_d(y), 
\end{equation} 
where $\Lambda^{\mu}_d$ is a constant diagonal
matrix, see for instance \cite{HvP} or \cite[Section 3]{vPR}. Note that the
expression (\ref{eq:Omega_1}) is available for any irreducible
$K$-representation, see e.g.~\cite[Prop.~9.1.2.11]{Warner2}. The full spherical
function $\Psi^{\mu}_0$ of degree zero is also a solution to the first-order
differential equation, \begin{equation} \label{eq:first_O_DE_Psi0}
y(1-y)\partial_y \Psi^{\mu}_0(y)=(S^{\mu}+yR^{\mu})\Psi^{\mu}_0(y),
\end{equation} for certain constant matrices $R^{\mu}$ and $S^{\mu}$,
\cite[Theorem 3.1]{vPR}. Note that differentiating (\ref{eq:first_O_DE_Psi0})
yields a matrix-valued hypergeometric differential equation for
$\Psi_{0}^{\mu}$ in the sense of \cite{TiraPNAS}. Observe that
\eqref{eq:first_O_DE_Psi0} can be seen as a differential operator acting on the
left of $\Psi^{\mu}_0$ while \eqref{eq:Psi_d_eig_Omega} is an operator acting on
the right. One of the consequences of \eqref{eq:first_O_DE_Psi0} is that
$\Psi^{\mu}_0$ is invertible on $(0,1)$, see \cite[Cor.~3.4]{vPR} so that
\begin{equation*}
P^{\mu}_d(y)=\Psi^{\mu}_d (y) \,
(\Psi^{\mu}_0(y))^{-1},
\end{equation*}
Now that the link between the spherical functions of type $\mu$ and
the corresponding MVOPs is clear, we allow ourselves to drop the index $\mu$
from the notation. However, various constants and coefficients that occur later
on, do depend on $\mu$.

The matrix-valued polynomials $P_d$ satisfy a three-term recurrence relation,
\begin{equation}
\nonumber
yP_d(y)=A_d P_{d+1}(y) + B_d P_{d}(y)
+ C_d P_{d-1}(y),
\end{equation} with $A_d$ being an invertible diagonal matrix
so that the leading coefficient of $P_d$ is invertible for all
$d\in\mathbb{N}_0$, see \cite[Section 1]{HvP}. Furthermore the orthogonality
relations for the full spherical functions \eqref{eq:orthogonality_Psi} imply
that the polynomials $P_d$ are orthogonal with respect to the pairing
\begin{equation}
\nonumber
\langle P,Q\rangle_W=\int_0^1 P(y)W(y)Q(y)^*dy, \end{equation} where $W$ is the weight
matrix defined by 
\begin{equation} 
\label{eq:structure_W_group_level}
W(y)=(1-y)^\alpha y^\beta\, W_{pol}(y), \qquad W_{pol}(y)= \Psi_0(y) \, T \,
(\Psi_0(y))^*. 
\end{equation}

The first order differential equation \eqref{eq:first_O_DE_Psi0} implies that
the polynomials $P_d$ are eigenfunctions of the hypergeometric differential
operator 
\begin{equation} 
\label{eq:intro Dmu} 
D=\Phi_0\, \Omega \,
(\Phi_0)^{-1}=y(y-1)\partial_{y}^2+\partial_y(C-yU)-V, 
\end{equation} 
where
$$C=\frac{\lambda_1m}{rp^2(M-m)}-2 S, \quad U=2R-\frac{\lambda_1}{rp^2},\quad
V=-\frac{\Lambda_0}{rp^2}.$$ 
Here $M,m,p,r$ are constants related with the pair $(G,K)$ and their values 
are given in \cite[Table 2]{vPR} for the various cases.
The differential operator \eqref{eq:intro Dmu} is symmetric with respect to $W$
so that the pair $(W,D)$ is a MVCP. This fact follows from the symmetry of the
Casimir operator $\Omega$ on $G$.\\

\textbf{Deformation of MVOPs.}
We summarize the discussion above as follows: we have two pairs, a pair
$(w,\Omega)$ together with a sequence of full spherical functions $(\Psi_d)_{d}$
and a matrix valued classical pair  $(W,D)$ with a sequence of MVOPs
$(P_d)_{d}$. These pairs are related by the function $\Psi_{0}$, as follows:
\begin{equation*} 
W(y) = y^{\alpha}(1-y)^{\beta}  \Psi_0(y) \, T \,
(\Psi_0(y))^*,\qquad D=\Psi_0\, \Omega \, \Psi_0^{-1}, 
\end{equation*}
\begin{equation*} 
P_d=\Psi_d(\Psi_0)^{-1} \quad \forall d\in\mathbb{N}.
\end{equation*}
We shall refer to the functions, weights, differential operators on the
spherical level and on the polynomial level, where the relation is given by the
function $\Psi_{0}$.

The Jacobi polynomials in a single variable can be given as a family of scalar
valued orthogonal polynomials $(P^{(\alpha,\beta)}_{d})_{d}$, associated to the
scalar valued classical pair $(w^{(\alpha,\beta)},D^{(\alpha,\beta)})$. In this
case, there is a two dimensional deformation that allows for the fundamental
shifts \begin{itemize} \item $G_{+}(\alpha,\beta)=\del_{y}$, the shift is
$(1,1)$, \item
$G_{-}(\alpha,\beta)=2y(y-1)\del_{y}+(\alpha+\beta)(2y-1)+\alpha-\beta$, the
shift is $(-1,-1)$, \item $E_{+}(\alpha,\beta)=y\del_{y}+\beta$, the shift is
$(1,-1)$, \item $E_{-}(\alpha,\beta)=(y-1)\del_{y}+\alpha$, the shift is
$(-1,1)$. \end{itemize} This means, for example, that
$(E_{+}(\alpha,\beta)P^{(\alpha,\beta)}_{d+1})_{d}$ is a family of scalar valued
orthogonal polynomials associated to
$(w^{(\alpha+1,\beta-1)},D^{(\alpha+1,\beta-1)})$. The Jacobi polynomials with
geometric parameters (i.e.~$(\alpha,\beta)$ is associated to the root
multiplicities of compact symmetric spaces) are examples of MVOPs, where
$\Psi_{0}=1$. The theory of the shift operators is now an application of the
Heckman-Opdam theory on the level of the spherical functions, and it translates
in the above fashion to the shift operators of the Jacobi polynomials. This
transition becomes more involved if we consider more general $K$-types, because
$\Psi_{0}$ is then no longer trivial. This is why we content to study only
deformations that allow for the shift operator $\partial_{y}$. We proceed in
three steps.

\begin{itemize}
	
\item[\textbf{D1.}] The deformed differential operators
$D^{(\kappa)}$ have to satisfy
$D^{(\kappa+1)}\circ\partial_{y}=\partial_{y}\circ D^{(\kappa)}$. In this way,
whenever $(P_{d}^{(\kappa)})_{d}$ is a sequence of eigenfunctions of
$D^{(\kappa)}$, then so is $(\partial P_{d}^{(\kappa)})_{d}$ for
$D^{(\kappa+1)}$, with $\kappa\in\NN_{0}$. The dependence on $\kappa$ is polynomial,
hence we may vary $\kappa$ in $\CC$ by analytic continuation.

\item[\textbf{D2.}] The deformed differential operators $D^{(\kappa)}$ give rise
differential operators on the level of spherical functions, 
\begin{eqnarray}\nonumber
\Omega^{(\kappa)}:=\Psi_{0}^{-1}\circ
D^{(\kappa)}\circ\Psi_{0}=y(1-y)\partial_y^2 +
a^{(\kappa)}(y)\partial_y+F^{(\kappa)}(y). 
\end{eqnarray}
We stipulate that
$\Omega^{(\kappa)}$ be symmetric with respect to deformed weight
\begin{eqnarray}
\label{eqn: deformed weight}
w^{(\kappa)}(y)dy=T^{(\kappa)}y^{\alpha(\kappa)}(1-y)^{\beta(\kappa)}dy,
\end{eqnarray} 
with $T^{(\kappa)}>0$ diagonal. This yields a equations for
$T^{(\kappa)},a^{(\kappa)},F^{(\kappa)},\alpha(\kappa)$ and $\beta(\kappa)$.

\item[\textbf{D3.}] 
The question remains, whether the pair $(W^{(\kappa)},D^{(\kappa)})$ 
allows for the shift $\partial_y$. In other words, if $(P^{(\kappa)}_{d})_{d}$ is a
family of MVOPs with respect to $W^{(\kappa)}$ we want to decide wether the sequence of derivatives
 $(\partial_y P^{(\kappa)}_{d})_{d}$ is orthogonal with respect to $W^{(\kappa+1)}$. 
 Here we use criterion from Cantero, Moral and Vel\'azquez \cite{CantMV}, who characterized the sequences
MVOPs whose derivatives yield a families of MVOPs. 

\end{itemize}

Note that \textbf{D1} implies that $a^{(\kappa)}(y)$ is scalar valued function.
We can now state our main result. Let $(G,K,\mu)$  be a triple  as in
\cite{HvP}, and let $\Omega$ be the radial part of the Casimir operator
\eqref{eq:Omega_1}. Let $T$ be the diagonal matrix introduced in
\eqref{eq:orthogonality_Psi}. The associated matrix-valued classical pair
$(W,D)$ is given by (\ref{eq:structure_W_group_level}) and (\ref{eq:intro Dmu}).
\begin{theorem} 
\label{thm:deformation} For $\kappa\geq0$ let
$a^{(\kappa)}(y)=a(y)+\kappa(1-2y) = \beta+\kappa+1-y(\alpha+\beta+2\kappa+2)$,
let $F^{(\kappa)}$ be given by
$$F^{(\kappa)}=F-\kappa\Psi_0^{-1}(U+\kappa-1)\Psi_0-\kappa(1-2y)\Psi_0^{-1}\Psi_0'.$$
and let $T^{(\kappa)}$ be a solution of
\begin{equation}
\label{eq:conditionT}
T^{(\kappa)}\,(F^{(\kappa)}(y))^*=F^{(\kappa)}(y)\,T^{(\kappa)}. 
\end{equation}
Then the pair \begin{equation} \label{eq:Definition_Wnu_general}
W^{(\kappa)}(y)=(1-y)^{\alpha+\kappa}y^{\beta+\kappa} \, \Psi_0(y) \,
T^{(\kappa)} \, \Psi_0(y)^*, \qquad D^{(\kappa)}=\Psi_0^{-1}\,
\Omega^{(\kappa)}\, \Psi_0, \end{equation} is a MVCP, if $T^{(\kappa)}>0$.
\end{theorem}

\begin{remark} Using \textbf{D3} we show that several examples of MVCPs that we
obtained in \cite{vPR} admit a deformation that allows for the shift
$\partial_{x}$. Moreover, the deformations of the MVCPs of \cite{KvPR,KvPR2}
that we obtain, are the same, up to a constant, as those in \cite{KdlRR}. \end{remark}

\begin{remark} \label{rmk:weight_positive} Note that $W^{(\kappa)}>0$ almost
everywhere if and only if $T(\kappa)>0$. The conclusion that $W^{(\kappa)}>0$
almost everywhere for the case studied in \cite{KdlRR}, depends on the
LDU-decomposition of the weight. The method that we propose has the important
advantage that its nature is simpler than the deformation considered in
\cite{KdlRR}, because we go back to the spherical level, where the weight and
the differential operator are in some sense much simpler.

In fact, the matrix $T=T^{0}$ corresponding to $(G_{0},K_{0},\mu_{0})$ is a
diagonal matrix whose entries are dimensions of the irreducible $M_{0}$-modules
that occur in the restriction of $\pi^{K_{0}}_{\mu_{0}}$ to $M_{0}$ (where
$M_{0}\subset K_{0}$ is a compact subgroup that we obtain after some choices,
see \cite{vPR}). In some of the examples, where the $T^{(k)}$ can be interpreted
as the $"T^{(0)}"$ for other triples $(G_{k},K_{k},\mu_{k})$, our deformation is
really an interpolation of all these dimensions. \end{remark}

%
The fact that we pass from the level of spherical functions to the level of
polynomials, always using the same function $\Psi_{0}$, is a restriction to the
theory. However, this restriction is justified by the examples, where we also
have families $(G_{k},K_{k},\mu_{k})_{k\in\NN}$ of Gelfand pairs with specified
$K_{k}$-type $\mu_{k}$, such that $\Psi_{0}^{\mu_{k}}=\Psi_{0}^{\mu_{k_{0}}}$,
for some $k_{0}\in\NN$.

A second restriction to the theory is given by the shape of the weight matrices
on the spherical level, that we insist to be of the form (\ref{eqn: deformed
weight}) with $T^{(\kappa)}>0$ diagonal. This is justified by the same argument
as before, namely that this is what happens in the examples.\\

\textbf{Outlook.} It would be interesting to investigate the effect of other
shift operators on the MVOPs in a single variable. A possible approach is to
study the group theoretic interpretation of the shift operators for MVOPs with
geometric parameters. This may also give more insight to the result of Cantero,
Moral and Vel\'azquez \cite{CantMV} on the level of the spherical functions.

It would also be interesting to determine the generators for the algebra of
differential operators that have the spherical functions as eigenfunctions. If
the dependence of these generators on the root multiplicities is understood, one
should investigate whether the whole algebra may be deformed, instead of just
one of the differential operators.\\

\textbf{Organization of the paper.} In Section \ref{sect:Deformation} we will
work out \textbf{D1}, \textbf{D2} and derive equations on the level of the
spherical functions that will in turn yield a deformation of the given MVCP. In
view of \textbf{D3} we explain how it can be verified that the deformation
allows for the shift $\partial_{y}$. For the cases where this holds true we
derive a Rodrigues type formula.

In Section \ref{sec:MVchebyshev} we apply our construction to the family of
matrix-valued Chebyshev polynomials obtained in \cite{KvPR,KvPR2}. In Section
\ref{sec:other_examples} we apply our new method to the examples of $2\times 2$
matrix-valued orthogonal polynomials obtained in \cite{vPR}. We see that the
families related to the groups $\SU(n)$ and $\SO(n)$ are maximal in the sense
that they are closed under our deformations. On the other hand, the families
associated to the symplectic groups $\Sp(n)$ give rise to a new families with an
extra free parameter.

\section{Deformation of MVCPs} \label{sect:Deformation}

In this section we carry out \textbf{D1}-\textbf{D3} of the introduction. Let
$(w,\Omega)$ and $\Psi_{0}$ be the input data coming from the representation
theory of compact Gelfand pairs of rank one, and let $(W,D)$ be the MVCP
associated to this data.

\begin{proposition} \label{lem:derivatives} Let $Q$ be a polynomial on $\RR$
that satisfies \begin{equation}\nonumber
y(1-y)Q''(y)+Q'(y)(C-yU)-Q(y)V=\Lambda Q(y), \end{equation} where $\Lambda$ is a
constant matrix. Then the $k$-th derivative $\partial_{y}^{k}Q=Q^{(k)}$
satisfies \begin{equation}\nonumber
y(1-y)(Q^{(k)})''(y)+(Q^{(k)})'(y)(C+k-y(U+2k))-Q^{(k)}(y)(V+kU+k(k-1))=\Lambda
Q^{(k)}(y). \end{equation} \end{proposition} \begin{proof} The verification,
which is based on induction over $k$, is left to the reader. \end{proof} The
following Corollary follows immediately by analytic continuation and it settles
\textbf{D1} of our program. \begin{corollary}\label{Cor: Deformation D} Given a
differential operator $D=y(1-y)\,\partial_y^2+\partial_y\,(C-yU)-V$, the
differential operator
$D^{(\kappa)}=y(1-y)\,\partial_y^2+\partial_y\,(C^{(\kappa)}-yU^{(\kappa)})-V^{(
\kappa)}$ with $C^{(\kappa)}=C+\kappa$, $U^{(\kappa)}= U+2\kappa$ and
$V^{(\kappa)}=V+\kappa U+\kappa(\kappa-1)$, satisfies $\partial\circ
D^{(\kappa)}=D^{(\kappa+1)}\circ\partial$ for all $\kappa\ge0$. \end{corollary}
The deformation $(w^{(\kappa)},\Omega^{(\kappa)})$ of $(w,\Omega)$ is of the
form 
\begin{equation}
\label{eq:Omega_nu_s2}
\Omega^{(\kappa)}=y(1-y)\partial_y^2+a^{(\kappa)}(y)\partial_y+F^{(\kappa)}(y),
\quad w^{(\kappa)}(y)=y^{\beta(\kappa)}(1-y)^{\alpha(\kappa)}T^{(\kappa)},
\end{equation} 
for certain functions $a^{(\kappa)}(y), F^{(\kappa)}(y)$, a
constant matrix $T^{(\kappa)}$ and constants $\alpha(\kappa),\beta(\kappa)$ such
that 
\begin{equation}\nonumber
a^{(0)}(y)=a(y),\qquad
F^{(0)}(y)=F(y),\qquad T^{(0)}=T,\qquad \alpha(0)=\alpha,\qquad\beta(0)=\beta.
\end{equation} 
This leads to a deformed pair
$(W^{(\kappa)},D^{(\kappa)})=(\Psi_0 \, w^{(\kappa)} \, \Psi_0^*, \Psi_0 \,
\Omega^{(\kappa)} \, \Psi_0^{-1})$. To see how this deformation translates to
the level of spherical functions, where it is easier to check the symmetry
conditions, we use the following result.

\begin{proposition} 
\label{prop:conditions_Omega_nu_hyperg} 
Let
$\Omega^{(\kappa)}$ be given by \eqref{eq:Omega_nu_s2}. The differential
operator $D^{(\kappa)}=\Psi_0\, \Omega^{(\kappa)} (\Psi_0)^{-1}$ is a
matrix-valued hypergeometric operator of the form
$$D^{(\kappa)}=y(1-y)\,\partial_y^2+\partial_y\,(C^{(\kappa)}-yU^{(\kappa)})-V^{(\kappa)},$$
for some constant matrices $C^{(\kappa)}, U^{(\kappa)}, V^{(\kappa)}$ if and
only if $a^{(\kappa)}$ and $F^{(\kappa)}$ are given by 
\begin{eqnarray}
F^{(\kappa)}(y) &= &-\Psi_0^{-1}\,V^{(\kappa)}\,\Psi_0 - y(1-y) \Psi_0^{-1}\,
\Psi_0''-a^{(\kappa)}\Psi_0^{-1}\Psi_0',\label{eq:Formula_F_nu}\\
a^{(\kappa)}(y) &=& (C^{(\kappa)}-yU^{(\kappa)}) -2(S+yR),\nonumber
\end{eqnarray} where $S, R$ are the matrices given in
\eqref{eq:first_O_DE_Psi0}. \end{proposition} \begin{proof} By a straightforward
computation, it is readily seen that $\Psi_0^{-1}\Omega^{(\kappa)}\Psi_0$ is the
differential operator \begin{multline*} y(1-y)\partial^2_y+\partial_y\,
[2y(1-y)\Psi_0'\Psi_0^{-1}+a^{(\kappa)}]+[y(1-y)\Psi_0''\Psi_0^{-1}+a^{(\kappa)}
\Psi_0'\Psi_0^{-1}+\Psi_0F^{(\kappa)}\Psi_0^{-1}]. \end{multline*} This is a
matrix-valued hypergeometric operator if and only if there exist constant
matrices $C^{(\kappa)}, U^{(\kappa)}$ and $V^{(\kappa)}$ such that
\begin{equation} 
\label{eq:Omega_D_hyperg}
2y(1-y)\Psi_0'\Psi_0^{-1}+a^{(\kappa)}= C^{(\kappa)}-yU^{(\kappa)},\quad
y(1-y)\Psi_0''\Psi_0^{-1}+a^{(\kappa)}\Psi_0'\Psi_0^{-1}+\Psi_0F^{(\kappa)}
\Psi_0^{-1}=V^{(\kappa)}. 
\end{equation} 
In the first equation of
\eqref{eq:Omega_D_hyperg} we use \eqref{eq:first_O_DE_Psi0} to obtain
$$2(S+yR)+a^{(\kappa)}(y) = (C^{(\kappa)}-yU^{(\kappa)}).$$ Finally the second
equation of \eqref{eq:Omega_D_hyperg} holds if and only if $F^{(\kappa)}$ is
given by \eqref{eq:Formula_F_nu}. \end{proof} As an immediate consequence of our
calculations, we obtain the following result, which in particular shows that
$a^{(\kappa)}$ is a scalar function, whenever we deform $D$ according to
Corollary \ref{Cor: Deformation D}.

\begin{corollary} 
\label{cor:a_scalar} 
The deformed differential operator is of the form
$$D^{(\kappa)}=y(1-y)\,\partial_y^2+\partial_y\,(C^{(\kappa)}-yU^{(\kappa)})-V^{
(\kappa)},$$ with $C^{(\kappa)}=C+\kappa$, $U^{(\kappa)}= U+2\kappa$ and
$V^{(\kappa)}=V+\kappa U+\kappa(\kappa-1)$ if and only if
$$a^{(\kappa)}(y)=a(y)+\kappa(1-2y),\qquad
F^{(\kappa)}(y)=F(y)+\kappa\Psi_0^{-1}(U+\kappa-1)\Psi_0-\kappa(1-2y)\Psi_0^{-1}
\Psi_0'.$$ \end{corollary}

We proceed to investigate the symmetry relations for the deformed differential
operator on the level of the spherical functions.

\begin{proposition} \label{prop:symmetry_Omega_nu} Assume that the differential
operator $D^{(\kappa)}= \Psi_0\, \Omega^{(\kappa)} (\Psi_0)^{-1}$ is of the form
of Corollary  \ref{cor:a_scalar}, then $D^{(\kappa)}$ is symmetric with respect
to $W^{(\kappa)}$ if and only if \begin{equation} \label{eq:symmetry_Omega_nu_2}
(y(1-y)\,w^{(\kappa)}(y))'=a^{(\kappa)}(y)\,w^{(\kappa)}(y),\quad
T^{(\kappa)}\,(F^{(\kappa)}(y))^*=F^{(\kappa)}(y)\,T^{(\kappa)}. \end{equation}
\end{proposition} \begin{proof} Let $(P_d)$ be a sequence of MVOPs with respect
to $W^{(\kappa)}$ and let $\Psi_d^{(\kappa)}=P_d\Psi_0$. The operator
$D^{(\kappa)}$ is symmetric with respect to $W^{(\kappa)}$ if and only if
\begin{align*} \int_0^1 (\Psi_d^{(\kappa)}\Omega^{(\kappa)})(y)\, T^{(\kappa)}
\, (\Psi_{d'}^{(\kappa)})^* (1-y)^{\alpha(\kappa)} y^{\beta(\kappa)} dy
&=\int_0^1 (P_d  D^{(\kappa)})(y) \, W^{(\kappa)}(y)  \, (P_{d'})^*  dy\\
&=\int_0^1 P_d(y) \, W^{(\kappa)}(y)  \, (P_{d'}D^{(\kappa)})^*  dy\\ &=\int_0^1
\Psi_d^{(\kappa)}(y)\, T^{(\kappa)} \, (\Psi_{d'}^{(\kappa)}\Omega^{(\kappa)})^*
(1-y)^{\alpha(\kappa)} y^{\beta(\kappa)} dy, \end{align*} where we used the
symmetry condition \eqref{eq:Symmetry_D_Pd} in the second and third equation. In
other words, $D^{(\kappa)}$ is symmetric with respect to $W^{(\kappa)}$ if and
only if 
\begin{equation}
\label{eq:symmetry_Omega_Psid}
\int_0^1 (\Psi_d^{(\kappa)}\Omega^{(\kappa)})(y)\, T^{(\kappa)} \, (\Psi_d^{(\kappa)})^*
(1-y)^{\alpha(\kappa)} y^{\beta(\kappa)}dy = \int_0^1 \Psi_d^{(\kappa)}(y)\,
T^{(\kappa)} \, (\Psi_d^{(\kappa)}\Omega^{(\kappa)})^* (1-y)^{\alpha(\kappa)}
y^{\beta(\kappa)} dy. 
\end{equation} 
It follows from integration by parts that 
\eqref{eq:symmetry_Omega_Psid} holds true for all $d,d'\in \mathbb{N}$ if and
only if 
\begin{align} 
\label{eq:symmetry_Omega_nu_1}
(y(1-y)\,w^{(\kappa)}(y))'=a^{(\kappa)}(y)\,w^{(\kappa)}(y),\qquad
T^{(\kappa)}\,(F^{(\kappa)}(y))^*=F^{(\kappa)}(y)\,T^{(\kappa)}. 
\end{align}
This completes the proof of the proposition.
\end{proof}

\begin{remark} \label{rmk:TF=F*T_tridiag} If $F^{(\kappa)}$ is a tridiagonal
matrix and $F^{(\kappa)}_{i,i+1}/F^{(\kappa)}_{i+1,i}$ is constant for all
$i=0,\ldots,N$, then then the condition on the right of
\eqref{eq:symmetry_Omega_nu_2} is equivalent to
$$T^{(\kappa)}_{i+1,i+1}=\frac{F^{(\kappa)}_{i+1,i}}{F^{(\kappa)}_{i,i+1}} \,
T^{(\kappa)}_{i,i},$$ so that  $T^{(\kappa)}$ is determined up to a constant
factor. \end{remark}

Note that (\ref{eq:symmetry_Omega_nu_1}) for $\kappa=0$ implies
$a(y)=1+\beta-(2+\beta-\alpha)y$, so by (\ref{eq:symmetry_Omega_nu_2}) we have
$a^{\kappa}(y)=1+\beta+\kappa-(2+2\kappa+\beta-\alpha)y$. Again by
(\ref{eq:symmetry_Omega_nu_2}) we find $\alpha(\kappa)=\alpha+\kappa$ and
$\beta(\kappa)=\beta+\kappa$.

We have now settled step \textbf{D2} of our program. The question remains,
whether the deformation $(W^{(\kappa)},D^{(\kappa)})$ allows for the shift $\partial_y$.
In other words, we need to determine whether the  sequence of derivatives $(\partial_y P_n^{(\kappa)})$ is orthogonal with respect to $W^{(\kappa+1)}$. In \cite{CantMV} Cantero, Moral, Vel\'azquez
characterized the sequences of MVOPs whose derivatives are also
orthogonal. If there exist matrix polynomials
$\Gamma_2^{(\kappa)}$ and $\Gamma_1^{(\kappa)}$  of degree two and one
respectively such that \begin{equation} \label{eq:relation_W_Gamma2}
(W^{(\kappa)}(y)\Gamma^{(\kappa)}_2(y))'=W^{(\kappa)}(y)\Gamma_1^{(\kappa)}(y),
\end{equation} then the sequence of derivatives $(\partial_y P_n^{(\kappa)})$ is
orthogonal with respect to $W^{(\kappa)}(y)\Gamma_2^{(\kappa)}$. In our case,
using the expression \eqref{eqn: deformed weight} for the weight matrices, we
have that if \begin{align} \label{eq:def_Gamma2}
\Gamma_2^{(\kappa)}(y)&=y(1-y)(\Psi_0(y)^*)^{-1} \,
(T^{(\kappa)})^{-1}T^{(\kappa+1)} \, \Psi_0(y)^*, \\ \label{eq:def_Gamma1}
\Gamma_1^{(\kappa)}(y)&=(W^{(\kappa)}(y))^{-1}(W^{(\kappa+1)}(y))' \end{align}
are polynomials of degree two and one respectively, then sequence $(\partial_y
P_n^{(\kappa)})$ is orthogonal with respect to
$W^{(\kappa+1)}=W^{(\kappa)}\Gamma_2^{(\kappa)}$. We can rewrite
\eqref{eq:relation_W_Gamma2} in terms of differential operators, as stated in
the following proposition which summarizes this discussion. 
\begin{proposition}
\label{eq:D^nu_Gamma2_Gamma1} 
The following are equivalent. \begin{enumerate}
\item The deformed pair $(W^{(\kappa)},D^{(\kappa)})$ allows for the shift $\partial_y$.
\item $\Gamma_2^{(\kappa)}$ is a matrix polynomial
of degree two and $\Gamma_1^{(\kappa)}$ is a matrix polynomial of degree one.
\item The differential operator 
\begin{equation*}
D^{(\kappa)}_{(\Gamma_2,\Gamma_1)}=\frac{d^2}{dy^2}
\Gamma_2^{(\kappa)}(y)^*+\frac{d}{dy} \Gamma_1^{(\kappa)}(y)^*
\end{equation*}
is symmetric with respect to $W^{(\kappa)}$ and has the polynomials
$P_d^{(\kappa)}$ as eigenfunctions. \end{enumerate} 
\end{proposition}
\begin{proof} 1 $\iff$ 2 follows from \cite{CantMV}.

2 $\Rightarrow$ 3: If the sequence $(\partial_y P_d^{(\kappa)})$ is orthogonal
with respect to $W^{(\kappa+1)}$, then the equations
\eqref{eq:relation_W_Gamma2} hold true for polynomials $\Gamma_2^{(\kappa)}$ and
$\Gamma_1^{(\kappa)}$ of degrees two and one respectively, see \cite[Thm.
3.14]{CantMV}. A simple computation shows that \eqref{eq:relation_W_Gamma2} implies
the conditions for the symmetry of  \cite[Thm. 3.1]{DuraG}, we omit the proof.

3 $\Rightarrow$ 1: Now assume that $D^{(\kappa)}_{(\Gamma_2,\Gamma_1)}$ is
symmetric with respect to $W^{(\kappa)}$ and has the polynomials
$P_n^{(\kappa)}$ as eigenfunctions. The proof is completely
analogous to \cite[Proposition 3.3]{KdlRR}. Since
$D^{(\kappa)}_{(\Gamma_2,\Gamma_1)}$ is symmetric, $\deg
\Gamma_2^{(\kappa)} = 2$, $\deg \Gamma_1^{(\kappa)} = 1$ (see
\cite[Proposition 2.6]{GrunT}) and the conditions 
\begin{align*}
(\Gamma_2^{(\kappa)}(y)^{*}W^{(\kappa)}(y))''-(\Gamma_1^{(\kappa)}(y)^{*}
W^{(\kappa)}(y))'=0, \,\, \lim_{y\to
0,1}\left(\Gamma_2^{(\kappa)}(y)^{*}W^{(\kappa)}(y)\right)'
-\Gamma_1^{(\kappa)}(y)^* W^{(\kappa)}(y)=0, 
\end{align*} 
hold. If we integrate
with respect to $y$ we obtain
$(\Gamma_2^{(\kappa)}(y)^{*}W^{(\kappa)}(y))'=\Gamma_1^{(\kappa)}(y)^{*}
W^{(\kappa)}(y)$ which implies \eqref{eq:relation_W_Gamma2}. \end{proof} We have now
settled all the steps \textbf{D1}-\textbf{D3} of our program. We lack a group
theoretical interpretation for the matrix polynomials $\Gamma^{(\kappa)}_2$ and
$\Gamma^{(\kappa)}_1$. \\

\textbf{Rodrigues formula.} Now we assume that  Proposition \ref{eq:D^nu_Gamma2_Gamma1} 
holds true. One of the main consequences is the existence of a Rodrigues formula for the
MVOPs. This fact was already noticed in \cite[Theorem 3.1, (iii)]{KdlRR} and the proof in our case follows the
same lines. We take the Hilbert
$\mathrm{C}^*$-module  $\mathcal{H}^{(\kappa)}$ given by the completion of the
space of matrix-valued orthogonal polynomials with respect to the matrix-valued
inner product $\langle\cdot,\cdot\rangle^{(\kappa)}$ defined by $W^{(\kappa)}$.
The lowering operator $\partial_y:\mathcal{H}^{(\kappa)}\to\mathcal{H}^{(\kappa+1)}$ is an
unbounded operator with dense domain and dense range. This operator has an adjoint, which is a raising operator preserving polynomials, that can be calculated explicitly in terms of $\Gamma_2^{(\kappa)}$ and $\Gamma_1^{(\kappa)}$.
\begin{lemma} 
Let $\kappa>0$ and define the first order
differential operator $\Xi^{(\kappa)}$ given by 
$$Q\Xi^{(\kappa)}=\partial_y
Q(\Gamma_2^{(\kappa)})^*+Q(\Gamma_1^{(\kappa)})^*.$$ 
Then $\langle \partial_y
P,Q\rangle^{(\kappa+1)}=-\langle P,Q\Xi^{(\kappa)}\rangle^{(\kappa)}$ for
matrix-valued polynomials $P$ and $Q$. 
\end{lemma}
\begin{proof}
This is analogous to \cite[Corollary 2.5]{KdlRR}.
\end{proof}

\begin{theorem} 
\label{thm:Rodrigues_formula} 
Let $\kappa>0$ and let $(Q_n^{(\kappa)}))_n$ be the sequence of monic orthogonal polynomials with respect to $W^{(\kappa)}$. Then there exists a constant matrices $G^{(\kappa)}_n$ 
such that
\begin{equation} 
\label{eq:Rodrigues_general} 
Q_n(y)=G_n^{(\kappa)}(\partial_y^n \, W^{(\kappa+n)}) \, (W^{(\kappa)}(y))^{-1}. 
\end{equation} 
for all $n\in\mathbb{N}$.
\end{theorem}
\begin{proof}
Is analogous to \cite[Theorem 3.1, (iii)]{KdlRR}.
\end{proof}

\begin{corollary} 
The following integral formula holds,
\begin{equation}
\nonumber
Q_n(y)=\frac{G_n^{(\kappa)}}{2\pi i} \, \left[\int_{\gamma(x)}
\frac{W^{(\kappa+n)}(z)}{(z-y)^n} \, dz\right] \, (W^{(\kappa)}(y))^{-1},
\end{equation} 
where $\gamma$ is a closed contour around $y$. 
\end{corollary}

\begin{proof} The proof follows by applying Cauchy's integral formula to
\eqref{eq:Rodrigues_general}. 
\end{proof}

\begin{corollary} 
The following relation holds for the monic orthogonal
polynomials $Q_n$ 
\begin{equation}
\label{eq:diferential_relation_Pn}
G^{(\kappa)}_n(G^{(\kappa)}_{n+1})^{-1}Q^{(\kappa)}_{n+1}=(\partial_y
Q^{(\kappa+1)}_n)(\Gamma_2^{(\kappa)})^*+Q^{(\kappa+1)}_n
(\Gamma^{(\kappa)}_1)^*. 
\end{equation} 
\end{corollary} 
\begin{proof} 
It follows from \eqref{eq:Rodrigues_general} that
$$Q_n^{(\kappa+1)}(y)W^{(\kappa+1)}(y)=G^{(\kappa)}_n(\partial_y^{n} \,
W^{(\kappa+n+1)}(y)).$$ 
If we take the derivative with respect to $y$ on both
sides of this equation we obtain 
\begin{equation}
\label{eq:derivative_QW}
\partial_y(Q^{(\kappa+1)}_nW^{(\kappa+1)})=G^{(\kappa)}_n(\partial_y^{n+1} \,
W^{(\kappa+n+1)}(y))=G^{(\kappa)}_n(G^{(\kappa)}_{n+1})^{-1}Q^{(\kappa)}_{n+1}
W^{(\kappa)}.
\end{equation} 
On the other hand we have 
\begin{equation}
\label{eq:derivative_QW2} 
\partial_y(Q^{(\kappa+1)}_nW^{(\kappa+1)})=(\partial_y
Q^{(\kappa+1)}_n)W^{(\kappa+1)}+Q^{(\kappa+1)}_n(\partial_y W^{(\kappa+1)}).
\end{equation} 
By combining \eqref{eq:derivative_QW} and
\eqref{eq:derivative_QW2} and multiplying by $(W^{(\kappa)})^{-1}$ on the right,
we obtain
$$G^{(\kappa)}_n(G^{(\kappa)}_{n+1})^{-1}Q^{(\kappa)}_{n+1}=(\partial_y
Q^{(\kappa+1)}_n)W^{(\kappa+1)}(W^{(\kappa)})^{-1}+Q^{(\kappa+1)}_n(\partial_y
W^{(\kappa+1)})(W^{(\kappa)})^{-1}.$$ 
Using \eqref{eq:def_Gamma2} and
\eqref{eq:def_Gamma1} gives the result. 
\end{proof}

\begin{corollary} 
The matrices $\Gamma^{(\kappa)}_2$ and $\Gamma_1^{(\kappa)}$
can be written in terms of the monic polynomials $P^{(\kappa)}_n$ and the
coefficients $G^{(\kappa)}_n$ in the following way 
\begin{align*}
(\Gamma^{(\kappa)}_1)^*&=G_0^{(\kappa)}(G_1^{(\kappa)})^{-1} Q_1^{(\kappa)}, \\
(\Gamma^{(\kappa)}_2)^*&=G_1^{(\kappa)}(G_2^{(\kappa)})^{-1}
Q_2^{(\kappa)}-G_0^{(\kappa+1)}(G_1^{(\kappa+1)})^{-1}
(\Gamma_1^{(\kappa+1)})^*(\Gamma_1^{(\kappa)})^*. \end{align*} 
\end{corollary}
\begin{proof} The corollary follows by evaluating \eqref{eq:diferential_relation_Pn} in $n=0$ and $n=1$. 
\end{proof}

\section{Matrix-valued Gegenbauer polynomials} 
\label{sec:MVchebyshev} 
The goal of this section is to deform the family of matrix-valued Chebyshev polynomials
obtained in  \cite{KvPR,KvPR2} from the study of spherical functions associated
to the pair $(G,K)=(\SU(2)\times\SU(2),\mathrm{diag } \,  \SU(2))$. We will show
that our construction is an alternative to the one given in  \cite{KdlRR} and
provides a different factorization for the weight matrix. The key factorization
in \cite{KdlRR} is the LDU decomposition of the weight matrix, which  plays a
fundamental role, for instance to show that the weight matrix is positive
definite, see \cite[Corollary 2.5]{KdlRR}. In our case, this is a direct
consequence of the decomposition \eqref{eq:Definition_Wnu_general} of the
deformed weight, as we noticed in Remark \ref{rmk:weight_positive}.

In this case, we have all the ingredients to perform the deformation explicitly.
For $\ell\in\frac12\mathbb{N}$, the full spherical functions $\Phi_d:
[0,4\pi]\to \mathbb{C}^{(2\ell+1)\times(2\ell+1)}$ were introduced in
\cite[Definition 2.2, Theorem 2.1 and (2.5)]{KvPR}. In particular, the full
spherical function of degree zero is given by 
\begin{multline} 
\label{def:Phi_0}
(\Phi_0(t))_{n,m}= \sum_{j_1=-\frac{n}{2}}^{\frac{n}{2}} \,
\sum_{j_1=-\frac{2\ell-n}{2}}^{\frac{2\ell-n}{2}} \delta_{-\ell+m,j_1+j_2}
\binom{n}{j_1+\frac{n}{2}}\binom{2\ell-n}{j_2+\frac{(2\ell-n)}{2}}\binom{2\ell}{
2\ell-m}^{-1} \,e^{i(j_2-j_1)t}. 
\end{multline}
 If we denote $\Psi_d(y)=\Phi_d(\arccos(1-2y))$, the full spherical polynomial of degree $d$
is given by 
\begin{equation*} 
P_d(y)= \Psi_d(y) (
\Psi_0(y))^{-1},\qquad y\in[0,1]. 
\end{equation*} 
For $\ell=0$, it boils down to
a $1\times 1$ matrix and it is a multiple of the Chebyshev polynomial of degree
$d$. The polynomials $(P_d)_{d}$ form a sequence of matrix-valued orthogonal
polynomials with non-singular leading coefficients, see \cite[Proposition
4.6]{KvPR}, with respect to the sesqui-linear pairing 
\begin{equation*}
\langle P,Q\rangle = \int_{0}^1 P(y)W(y)Q(y)^*
dy,\quad W(y)=y^{1/2}(1-y)^{1/2} \Psi_0(y) \Psi_0(y)^*. 
\end{equation*} 
The full
spherical functions $\Psi_d$ satisfy the following differential equation,
\begin{equation*} 
\Psi_d(y)\, \Omega  =
y(1-y)\Psi_d''(y) + a(y) \Psi_d'(y) + \Psi_d(y)\, F(y) = \Lambda_d \, \Psi_d(y),
\end{equation*} 
where $a(y)=\frac32 - 3y$, the eigenvalue is the diagonal matrix
matrix $(\Lambda_d)_{i,i}=-d(2\ell+2+d)I+i(2\ell-i),$ and \begin{multline}
\label{eq:F(t)_SU2} F(y)=\sum_{i=0}^{2\ell} \frac{2y(1-y)(\ell(\ell+2)-i^2+2\ell
i)-\ell(2i+1)+i^2}{2y(1-y)}\, E_{i,i} \\
+\sum_{i=1}^{2\ell}\frac{i(2\ell-i+1)(1-2y)}{4y(1-y)} \, E_{i,i-1}
+\sum_{i=0}^{2\ell-1} \frac{(i+1)(2\ell-i)(1-2y)}{4y(1-y)} E_{i,i+1}.
\end{multline} The differential operator $\Omega$ is the radial part of the
Casimir operator on $G$, see  \cite[Section 7.2]{KvPR2}. The function $\Psi_0$
satisfies the first order differential equation \begin{equation}
\label{eq:Psi0_in_y} 2y(1-y)\,\Psi_0'(y)+\left(S-\ell+2\ell y\right)\Psi_0(y)=0,
\end{equation} where $S$ is the tridiagonal matrix
\begin{equation}\label{eq: S}
S= \sum_{i=1}^{2\ell}
\frac{i}{2} E_{i,i-1} + \sum_{i=0}^{2\ell-1} \frac{(2\ell-i)}{2} E_{i,i+1},
\end{equation}
see \cite[Lemma 7.12]{KvPR2}. The polynomials $P_d$ are joint eigenfunctions of
the matrix-valued differential operators $D$ and $E$ given by 
\begin{equation*}
D\,=\,y(1-y)\frac{d^2}{dy^2}+\left(\frac{d}{dy}\right)(C-yU)-V,\qquad
E\,=\,\left(\frac{d}{dy}\right)(yB_1+B_0)+A_0, 
\end{equation*}
where the matrices
$C$, $U$, $V$, $B_0$, $B_1$ and $A_0$ are given by \begin{gather*}
C=-\sum_{i=0}^{2\ell} \frac{(2\ell-i)}{2} E_{i,i+1} +\sum_{i=0}^{2\ell}
\frac{(2\ell+3)}{2}E_{ii} - \sum_{i=0}^{2\ell} \frac{i}{2}E_{i,i-1},\quad
U=(2\ell+3)I,\\ V= -\sum_{i=0}^{2\ell} i(2\ell-i) E_{i,i} \quad
A_0=\sum_{i=0}^{2\ell} \frac{(2\ell+2)(i-2\ell)}{2\ell} E_{i,i}, \quad
B_1=-\sum_{i=0}^{2\ell} \frac{(\ell-i)}{\ell} E_{i,i}, \\
B_0=-\sum_{i=0}^{2\ell} \frac{(2\ell-i)}{4\ell} E_{i,i+1} +\sum_{i=0}^{2\ell}
\frac{(\ell-i)}{2\ell}E_{ii} + \sum_{i=0}^{2\ell} \frac{i}{4\ell}E_{i,i-1}.
\end{gather*}

The first order differential equation \eqref{eq:Psi0_in_y} can be used to derive
a simple and compact expression for $\Psi_0$ which will be crucial in the
forthcoming subsections. The proof is relegated to the Appendix. Let $K$ be the
constant matrix \begin{equation} \label{eq:K} K_{i,j}=K_j(i)=K_j(i,1/2,2\ell),
\qquad i,j\in \{0,\ldots,2\ell\}. \end{equation} where $K_n(x,p,N)$ are the
Krawtchouk polynomials, see e.g. \cite[\S1.10]{KoekS}. The orthogonality
relations for the Krawtchouk polynomials give a simple inverse for the matrix
$K$, namely $$K^{-1}=2^{-2\ell}MKM,$$ where $M$ is the diagonal matrix with
entries $M_{j,j}=\binom{2\ell}{j}$, $j=0,\ldots,2\ell$.
%
\begin{theorem} 
\label{prop:Psi_0_as_Krawtchouk} 
For any
$\ell\in\frac12\mathbb{N}$, we have 
\begin{equation}
\label{eq:decomposition_Phi0} \Psi_0(y)=K \Upsilon(y)K, 
\end{equation} where $K$
is the constant matrix given by \eqref{eq:K} and $\Upsilon$ is the diagonal
matrix 
\begin{equation*}
\Upsilon(y)_{j,j} = (-1)^{\frac{3j}{2}}
\binom{2\ell}{j}y^{\frac{j}{2}}(1-y)^{\frac{2\ell-j}{2}}.
\end{equation*}
\end{theorem}

\subsection{The deformation} In this subsection we apply Theorem
\ref{thm:deformation} to obtain a deformation of the pair $(W,D)$. Since we want
to deform matrix-valued Chebyshev polynomials into matrix-valued Gegenbauer
polynomials, we shift $\kappa=\nu-1$ in order to match the standard convention
for Gegenbauer polynomials. In this way, our deformed polynomials $P_d^{(\nu)}$
coincide with the polynomials $P_d$ for $\nu=1$. We take 
\begin{equation*}
a^{(\nu)} (y) = \frac12 +\nu - y(2\nu+1), \qquad
F^{(\nu)}=F-(\nu-1)(2\ell+\nu+1)-(\nu-1)(1-2y)\Psi_0^{-1}\Psi_0',
\end{equation*}
where $F$ is given in \eqref{eq:F(t)_SU2}. We follow the method described in Section
\ref{sect:Deformation}. It follows from the explicit expression in Theorem
\ref{prop:Psi_0_as_Krawtchouk}  that 
\begin{equation*}
\Psi_0^{-1}\Psi_0'=-\frac{1}{2y(1-y)}(S^*-\ell+2\ell y).
\end{equation*}
Therefore we have 
\begin{equation} 
\label{eq:Fnu_SU2}
F^{(\nu)}(y)=F(y)-(\nu-1)(2\ell+\nu+1)+\frac{(\nu-1)(1-2y)}{2y(1-y)}\left(S^*-\ell+2 \ell y\right). \end{equation} Note that $F^{(\nu)}$ is a tridiagonal
matrix. Therefore we can use Remark \ref{rmk:TF=F*T_tridiag} to obtain a
diagonal matrix $T^{(\nu)}$ as in Theorem \ref{thm:deformation}. \begin{lemma}
Let $T^{(\nu)}$ be the diagonal matrix \begin{equation}
\label{eq:formula_Tnu_SU2} T^{(\nu)}_{i,i}=\binom{2\ell}{i}
\frac{(\nu)_i}{(\nu+2\ell-i)_i}, \end{equation} for $i=0,\ldots,\lfloor \ell
\rfloor$ and $T_{i,i}=T_{2\ell-i,2\ell-i}$. Then 
\begin{equation*}
T^{(\nu)}\,(F^{(\nu)}(y))^*=F^{(\nu)}(y)\,T^{(\nu)}. 
\end{equation*} 
\end{lemma}
\begin{proof} It follows from \eqref{eq:Fnu_SU2} and \eqref{eq:F(t)_SU2} that
\begin{equation*}
F^{(\nu)}_{i,i-1}=(2\ell-i+1)(\nu-i+1)\frac{(1-2y)}{y(1-y)},\qquad
F^{(\nu)}_{i,i-1}=(i+1)(2\ell+\nu-i-1)\frac{(1-2y)}{y(1-y)}.
\end{equation*}
By Remark
\ref{rmk:TF=F*T_tridiag}, if we define $T^{(\nu)}$ as the diagonal matrix
\begin{equation*}
T^{(\nu)}_{i+1,i+1}=\frac{F^{(\nu)}_{i+1,i}}{F^{(\nu)}_{i,i+1}} \, T^{(\nu)}_{i,i} = \frac{(2\ell-i+1)(\nu-i-1)}{i(2\ell+\kappa-i)} \, T^{(\nu)}_{i,i},
\end{equation*} 
then the condition in Theorem \ref{thm:deformation} holds
true. This completes the proof of the lemma. \end{proof} \begin{corollary} If
$T^{(\nu)}$ is as in Lemma \ref{eq:formula_Tnu_SU2}, then the pair
\begin{align*} W^{(\nu)}(y)&=(1-y)^{\nu-1/2}y^{\nu-1/2} \, \Psi_0(y) \,
T^{(\nu)} \, \Psi_0(y)^*, \\
D^{(\nu)}&=y(1-y)\partial_y^2+\partial_y(C+\nu-1-y(U+2\nu+1))-(V+(\nu-1)
U+(\nu-1)(\nu-2)), \end{align*} is a MVCP. \end{corollary} \begin{remark}
Observe that the differential operator $2D^{(\nu)}$ is, up to a the change of
variables $x=1-2y$ and a multiple of the identity, the differential operator
$D^{(\nu)}$ in \cite[Theorem 2.3]{KdlRR}.  In the following subsection we will
show that the weight matrix $W^{(\kappa)}$ is also closely related to the weight
matrix introduced in \cite[Definition 2.1]{KdlRR}. \end{remark} \begin{remark}
In \cite{TZ} the authors construct families of matrix-valued orthogonal
polynomials of size $2\times2$ and $3\times 3$ from the study of spherical
functions of fundamental $K$-types associated with the pair $(G,K) =
(\SO(n+1),\SO(n))$. If we restrict our deformed weight $W^{(\kappa)}$ to the
size $3\times 3$, it can be matched with the one given in \cite[Section 9.2]{TZ}
by the identification $\kappa=1/(l+1)$. \end{remark}

\subsection{The shift operator $\partial_y$} 
The goal of this subsection is to prove that the pair $(W^{(\kappa)},D^{(\kappa)})$ allows for the
shift $\partial_y$. For this we will show that  $\Gamma_2^{(\nu)}$ and $\Gamma_1^{(\nu)}$
defined as is \eqref{eq:def_Gamma2} and \eqref{eq:def_Gamma1} are matrix-valued
polynomials of degree two and one respectively. Then  by Proposition
\ref{eq:D^nu_Gamma2_Gamma1}, if $(Q_n^{(\kappa)})_n$ is the sequence of monic
orthogonal polynomials with respect to $W^{(\kappa)}$, we have that $\partial_y
Q^{(\nu)}_d=dQ^{(\nu+1)}_d$, since the sequence of monic orthogonal polynomials
is unique. The proof will follow from the explicit expression of $\Psi_0$ given
in Theorem \ref{prop:Psi_0_as_Krawtchouk} and involves manipulation of
Krawtchouk polynomials. Some of the formulas for Krawtchouk polynomials that are
necessary in the proof are collected in the Appendix.

\begin{proposition} 
\label{prop:Gamma_2_Gamma1_poly} 
The functions
$\Gamma_2^{(\nu)}$ and $\Gamma_1^{(\nu)}$ introduced in  \eqref{eq:def_Gamma2}
and  \eqref{eq:def_Gamma1}  are matrix polynomials of degree two and one
respectively. 
\end{proposition} 

\begin{proof} First we compute the constant
matrix $\Delta=2^{-2\ell}MKM(T^{(\nu)})^{-1}T^{(\nu+1)}K$. Note that the $i$-th
diagonal element of  $(T^{(\nu)})^{-1}T^{(\nu+1)}$ is
$(\nu+2\ell-i)(\nu+i)/(\nu(\nu+2\ell))$, so that 
\begin{align*}
\Delta_{k,j}=\frac{2^{-2\ell}}{\nu(\nu+2\ell)}\, \sum_{i=0}^{2\ell}
\,\binom{2\ell}{i}\binom{2\ell}{k} \, (\nu+2\ell-i)(\nu+i) \, K_i(k) K_j(i).
\end{align*} 
It follows from the relations \eqref{eq:ortho_K},
\eqref{eq:ortho_K_i} and \eqref{eq:ortho_K_i2},  that 
\begin{multline}
\label{eq:expression_Delta} \nu(\nu+2\ell) \, \Delta=-\sum_{k=0}
\frac{(k+1)(k+2)}{4} \, E_{k,k+2}+\sum_{k=0}
(\ell(\ell-1/2)+k(k/2-\ell)+\nu(\nu+2\ell) )\, E_{k,k}\\ -\sum_{k=0}
\frac{(2\ell-k+1)(2\ell-k+2)}{4} \, E_{k,k-2}. 
\end{multline} 
First we show that
$\Gamma_2^{(\nu)}$ is a polynomial of degree two. It follows from Proposition
\ref{prop:Psi_0_as_Krawtchouk} that 
\begin{equation*}
\Gamma_2^{(\nu)}(y)=y(1-y)\, K^{-1} \,
\bar \Upsilon(y)^{-1} \, \Delta \, \bar\Upsilon(y) \, K.
\end{equation*}
It follows directly from the explicit expressions of $\Delta$ and $\Upsilon$ 
that $y(1-y)\bar \Upsilon(y)^{-1} \, \Delta \, \bar\Upsilon(y)$ is a polynomial 
of degree two and, therefore, so is $\Gamma_2^{(\nu)}$.

Now we prove that $\Gamma_1^{(\nu)}$ is a polynomial of degree one. Note that
\begin{multline} 
\label{eq:gamma1_SU2} 
K\,\Gamma_1^{(\nu)}(y)\,K^{-1}=
-\frac{(2y-1)(2\nu+1)}{2} \bar\Upsilon(y)^{-1} \Delta \bar \Upsilon(y)
-2^{-2\ell-1} \bar\Upsilon(y)^{-1} MKMNK \bar \Upsilon(y) \\ -
\frac{\ell(2y-1)}{2} \bar\Upsilon(y)^{-1} \Delta \bar\Upsilon(y) + y(1-y)
\bar\Upsilon(y)^{-1} \Delta  \bar\Upsilon(y)', 
\end{multline} 
where $N=(T^{(\nu)})^{-1} S^* T^{(\nu+1)}$. 
A simple computation shows that $N$ is the tridiagonal matrix given by
\begin{equation*}
N_{i,i-1}=\frac{i(\nu+2\ell-i)(\nu+2\ell-i+1)}{2\nu(\nu+2\ell)}, \quad
N_{i,i}=0,\quad N_{i,i+1}= \frac{(2\ell-i)(\nu+i)(\nu+i+1)}{2\nu(\nu+2\ell)}.
\end{equation*}
Using the explicit expressions of $N$ and $K$ we obtain: 
\begin{multline*}
2^{-2\ell}(MKMNK)_{k,j}=\frac{\binom{2\ell}{k}}{2\nu(\nu+2\ell)} \,
\sum_{i=0}^{2\ell} \binom{2\ell}{i} \left[ (2\ell-i)(\nu+i)(\nu+i+1) \,
K_k(i)K_j(i+1) \right. \\ \left. + i(2\ell+\nu-i)(2\ell+\nu-i+1)\,
K_k(i)K_j(i-1)\right] 
\end{multline*} 
In the formula above, we replace the terms
$(2\ell-i)K_j(i+1)$ and $iK_j(i-1)$ using Lemma \ref{lemma: rec kraw} and we get
an expression that is be evaluated using \eqref{eq:ortho_K_i} and
\eqref{eq:ortho_K_i2}. We obtain 
\begin{multline*} 
2^{-2\ell-1}(MKMNK)_{k,j} =
\frac{(2\nu +3\ell-k+1)(k+1)(k+2)}{8\nu(\nu+2\ell)} \, \delta_{j,k+2}   \\
+\frac{(k-\ell)(k^2-2\ell
k-l+1-4\nu\ell-2\nu^2-2\ell^2)}{4\nu(\nu+2\ell)}\delta_{j,k}
-\frac{(2\ell-k+1)(2\ell-k+2)(\ell+2\nu+k-1)}{8\nu(\nu+2\ell)}\delta_{j,k-2}.
\end{multline*} 
It follows from \eqref{eq:gamma1_SU2} and the equation above
that $(K\,\Gamma_1^{(\nu)}(y)\,K^{-1})_{i,j}=0$ unless $j=k-2,k,k+2$. Moreover,
a straightforward computation using the explicit expressions of $\Delta$,
$\Upsilon$ and $2^{-2\ell}MKMNK$ shows that
$(K\,\Gamma_1^{(\nu)}(y)\,K^{-1})_{i,j}$  is a polynomial of degree one. This
completes the proof of the proposition. 
\end{proof} 

In order to relate the deformed family $W^{(\nu)}$ with the family 
introduced in \cite{KdlRR} we need the following corollary.

\begin{corollary} 
The polynomial $\Gamma_2^{(\nu)}$ is given explicitly by
\begin{multline*} 
\frac{4\kappa(\kappa+2\ell)}{\ell^2}\, \Gamma_2^{(\nu)}(x)=
(1-2y)^{2}\sum_{i=0}^{2\ell} \frac{\left( \ell-i\right)^{2}}{\ell^2}
E_{i,i}-4y(1-y) \frac{(\ell+\nu)^2}{\ell^2}\\
+(1-2y)\sum_{i=1}^{2\ell}\frac{(i-1-2\ell)(2\ell-2i+1)}{2 \ell^2} E_{i,i-1}
+(1-2y)\sum_{i=0}^{2\ell-1}\frac{(i+1)(2\ell-2i-1)}{2 \ell^2} E_{i,i+1}  \\
+\sum_{i=2}^{2\ell}\frac{(2\ell-i+2)(2\ell-i+1)}{4\ell^2} E_{i,i-2}
+\sum_{i=0}^{2\ell}\frac{- i (2\ell-i+1)-(2\ell-i)(i+1)}{4 \ell^2} E_{i,i}\\ +
\sum_{i=0}^{2\ell-2}\frac{(i+2)(i+1)}{4 \ell^2} E_{i,i+2}. 
\end{multline*}
\end{corollary} 

\begin{remark} Observe that, up to the change of variables
$x=1-2y$ and a constant $4\kappa(\kappa+2\ell)/\ell^2$, the matrix $\Gamma_2$
coincides with the  polynomial of degree two $\Phi$ given in
\cite[(4.9)]{KdlRR}. 
\end{remark} 

\begin{proof} 
The corollary is proven by a straightforward computation. From the proof of Proposition
\ref{prop:Gamma_2_Gamma1_poly}, we have that 
\begin{equation}
\label{eq:Gamma2_proof} K\, \Gamma_2^{(\nu)}(y)=y(1-y) \, \bar \Upsilon(y)^{-1}
\, \Delta \, \bar\Upsilon(y) \, K. 
\end{equation} 
Now the proof follows by a
tedious but direct verification of \eqref{eq:Gamma2_proof}. Using the explicit
expression of $\Delta$ given in \eqref{eq:expression_Delta} the right hand side
of \eqref{eq:Gamma2_proof} becomes 
\begin{multline*} 
y(1-y) \bar \Upsilon(y)^{-1} \, \Delta \, \bar\Upsilon(y)=\sum_{j=2}^{2\ell} y^2
\frac{(2\ell-j-1)(2\ell-j)}{4\nu(\nu+2\ell)} \, E_{j,j+2}\\ +\sum_{j=0}^{2\ell}
y(1-y) \frac{(\ell(\ell-1/2)+j(j/2-\ell)+\nu(\nu+2\ell))}{\nu(\nu+2\ell)} \,
E_{j,j} +\sum_{j=0}^{2\ell-2} (1-y)^2 \frac{j(j-1)}{4\nu(\nu+2\ell)} \, E_{j,j-2}.
\end{multline*} 
Therefore, both sides of \eqref{eq:Gamma2_proof} are
polynomials of degree two. The proof of the corollary follows by showing that
the coefficients of degree 0,1,2 on the left and right hand sides of
\eqref{eq:Gamma2_proof} coincide. We include a sketch of the proof for the
coefficient of degree two and the other cases are analogous. The coefficient of
$y^2$ of the $(j,k)$-th entry of \eqref{eq:Gamma2_proof} is given by
\begin{multline*} 
-(\nu+j)(\nu+2\ell-j)K_j(k)=\frac{j(j-1)}{4}\,
K_{k}(j-2)-(\ell(\ell-1/2)+j(j/2-\ell)+\nu(\nu+2\ell))K_{k}(j)\\
+\frac{(2\ell-j-1)(2\ell-j)}{4}\, K_{k}(j+2). 
\end{multline*} 
Now we apply \eqref{eq:diff_Krawt1} and \eqref{eq:diff_Krawt2} twice on the first and last
terms of the right hand side, and the three term recurrence relation on the left
hand side and the middle term of the right hand side. The equation
above becomes an expression involving Krawtchouk polynomials $K_{h}(j)$ for
$h=k-2,k-1,k,k+1,k+2$ with coefficients which are independent of $j$. The
equation is then verified by checking the coefficients of the Krawtchouk
polynomials of different degrees. The remaining cases follow in an similar way.
\end{proof} 

Recall that our deformed family of weight matrices $W^{(\nu)}$
coincides with the one given in \cite[Definition 2.1]{KdlRR} for $\nu=1$.  Since
 $W^{(\nu+1)}=W^{(\nu)}\Gamma_2^{(\nu)}$, see \eqref{eq:def_Gamma2}, in view of
\cite[Theorem 2.4]{KdlRR}, it follows that the weight matrix in \cite[Definition
2.1]{KdlRR} coincides up to a constant with $W^{(\nu)}$ for all integer values
of $\nu$. From a continuation argument we conclude that the two families of
weights coincide up to constant multiple.

\begin{remark} 
Note that the operator $E^{(\nu)}=E+\nu(A_{0}+B_{1})$ satisfies
$\partial\circ E^{(\nu)}=E^{(\nu+1)}\circ\partial$. Since
$[E^{(\nu+1)},D^{(\nu+1)}]\circ\partial=\partial\circ[E^{(\nu)},D^{(\nu)}]$, and
$[E,D] =0$ see \cite{KvPR,KvPR2}, it follows that $[E^{(\nu)},D^{(\nu)}] = 0$
for all $\nu\in \mathbb{N}$. Moreover for any $\nu\geq0$ and any smooth
$\mathbb{C}^{2\ell+1}$-valued function $F$ we have 
\begin{multline}
\label{eq:commutator_ED} F(D^{(\nu)}E^{(\nu)}-E^{(\nu)}D^{(\nu)}) =
\nu(1-2y)(\partial_y F)E+\nu(FE)(U+\nu-1) \\ +
\nu(FD)(A_0+B_1)-\nu(F(A_0+B_1))D-\nu(1-2y)\partial_y\nu(FE)-\nu(FE)(U+\nu-1).
\end{multline} 
Now \eqref{eq:commutator_ED} is a polynomial function in $\nu$
which is zero for infinitely many values of $\nu$ and thus it  is zero for all
$\nu$. Therefore we have $$[E^{(\nu)},D^{(\nu)}] = 0,\qquad \text{for all
}\nu>0.$$ Since $E^{(\nu)}$ commutes with $D^{(\nu)}$, it follows that the monic
orthogonal polynomials $P^{(\nu)}_n$ are eigenfunctions of $E^{(\nu)}$, i.e.
\begin{equation*}
P^{(\nu)}_nE^{(\nu)} = \Lambda_n(E^{(\nu)})P_n^{(\nu)},\qquad \text{for all
}n\in\mathbb{N},
\end{equation*} 
where $\Lambda_n(E^{(\nu)}) = nB_1+A_0+\nu(B_1+A_0)$. Observe
that since the eigenvalues $\Lambda_n(E^{(\nu)})$ are diagonal with real
entries, and thus Hermitian, the differential operator $E^{(\nu)}$ is symmetric
with respect to $W^{(\nu+1)}$, see \cite[Corollary 4.5]{GrunT}. Observe that the
differential operator $E^{(\nu)}$ coincides with the differential operators in
\cite[Corollary 4.1]{KdlRR}. 
\end{remark}


\section{Examples of dimension two} 
\label{sec:other_examples} 

In this section we apply our method to two of the examples given in \cite[\S 1]{vPR}. 
These examples are related to the Gelfand pairs $(\mathrm{SU}(n+1),\mathrm{U}(n))$
and $(\mathrm{USp}(2n),\mathrm{USp}(2n-2)\times\mathrm{USp}(2))$ and correspond to 
matrix-valued orthogonal polynomials of Jacobi type.

\subsection{Case {\bf a1}} Let us consider the Gelfand pair
$(G,K)=(\mathrm{SU}(n+1),\mathrm{U}(n))$. This example is studied in \cite[Page
7, case {\bf a1}]{vPR}. For $n\geq 2$ we have $\alpha=n-1$, $\beta=0$  and two
free parameters $1\leq i\leq n-1$ and $m\in \mathbb{N}$. We have the following
expressions, \begin{equation} \label{eq:Psi_0_Case_a2}
\Psi_0^{(n,m,i)}(y)=y^{\frac{m}{2}}\begin{pmatrix} y^{\frac12} & 1 \\
y^{\frac12} & \frac{(m+1)-y(m+n-i+1)}{i-n} \end{pmatrix}, \qquad T=
\begin{pmatrix} 1 & 0 \\ 0 & \frac{n-i}{i} \end{pmatrix}, \end{equation} where
we indicate the free parameters as superscripts. Moreover, we have
$a(y)=1-y(n+1)$ and $$F(y)=\begin{pmatrix} (1 + m)(1 + m + 2 n) +
\frac{(i-n)}{(1 - y)} - \frac{(1 + m)^2}{4 y} & \frac{i\sqrt{y}}{1-y} \\
\frac{(i-n)\sqrt{y}}{1-y} & \frac{m(y(m + 2n) - m)}{4y} -\frac{4 i y}{1 - y}
\end{pmatrix}$$
and the differential operator $D=y(1-y)\,
\partial^2_y+\partial_y, (C-yU)-V$ is determined by \begin{align*} C&=\begin{pmatrix}
{\frac {(m+1)(-m-n-2+i) }{-n-m-1+i}}&{\frac {-n+i}{-n-m-1+i}}\\ -{\frac
{m+1}{-n-m-1+i}}&{\frac {-{m} ^{2}+2i-2n+mi-2m-mn-1}{-n-m-1+i}} \end{pmatrix},
\qquad U=\begin{pmatrix} n+m+2 & 0 \\ -1 & n+m+3 \end{pmatrix}, \\
V&=\begin{pmatrix} 0 & 0 \\ 0 & n+m+1-i \end{pmatrix}. \end{align*} As in
Theorem \ref{thm:deformation}, we define
$F^{(\kappa)}=F+\kappa\Psi_0^{-1}(U+\kappa-1)\Psi_0-\kappa(1-2y)\Psi_0^{-1}
\Psi_0'$. By a straightforward computation we verify that
$$F^{(\kappa)}(y)_{0,1}= -{\frac {\sqrt {y} \left( \kappa+i \right)
}{y-1}},\qquad F^{(\kappa)}(y)_{1,0} = \frac {\sqrt {y} \left( -n+i \right)
}{y-1},$$ so that by Remark \ref{rmk:TF=F*T_tridiag}, the matrix $T^{(\kappa)} =
\begin{pmatrix} 1 & 0 \\ 0  & \frac{n-i}{i+\kappa} \end{pmatrix},$ is a solution
to \eqref{eq:conditionT}. Therefore Theorem \ref{thm:deformation} implies
that $(W^{(n,m,i,\kappa)},D^{(n,m,i,\kappa)})$ with 
\begin{align}
\label{eq:formal_W_SUn_nu} 
W^{(n,m,i,\kappa)}(y)&=y^{\kappa}(1-y)^{n-1+\kappa}
\Psi^{n,m,i}_0(y) \, T^{(\kappa)} \, (\Psi^{n,m,i}_0(y))^* \\
D^{(n,m,i,\kappa)}&=y(1-y)\partial_y^2+\partial_y(C+\kappa-y(U+2\kappa))-(V+
\kappa U+\kappa(\kappa-1)), \nonumber
\end{align} 
is a MVCP. It is a
straightforward computation that the functions
$$\Gamma_2^{(\kappa)}=(\Psi_0(y)^*)^{-1} \, (T^{(\kappa)})^{-1}T^{(\kappa)} \,
\Psi_0(y)^*,\qquad
\Gamma_1^{(\kappa+1)}=(W^{(n,m,i,\kappa)}(y))^{-1}(W^{(n,m,i,\kappa+1)}(y))',$$
are matrix-valued polynomials of degrees two and one respectively. We omit this
computation. It then follows that Proposition \ref{eq:D^nu_Gamma2_Gamma1} and
Theorem \ref{thm:Rodrigues_formula} apply to this case. Therefore the
monic orthogonal polynomials $(Q_d^{(\kappa)})_d$ with respect to $W^{(\kappa)}$ satisfy 
$$\partial_y \, Q^{(n,m,i,\kappa)}_d(y)=d\, Q_{d-1}^{(n,m,i,\kappa+1)}(y),$$ 
and the following Rodrigues formula holds 
$$Q_d(y)= G_d \, (\partial_y^d \, W^{(n,m,i,\kappa+d)})
\, (W^{(n,m,i,\kappa)})^{-1},$$ 
for certain constant matrix $G_d$.

\begin{remark} For nonnegative integer values of $\kappa$, it follows directly
from the form of $T^{(\kappa)}$ and \eqref{eq:formal_W_SUn_nu} that
$W^{(n,m,i,\kappa)}=W^{(n+\kappa,m+\kappa,i+\kappa)}$ so that 
$$\partial_y \, Q^{(n,m,i)}_d(y)=d\, Q_{d-1}^{(n+1,m+1,i+1)}(y).$$ 
\end{remark}

\begin{remark} 
The case $m\in \mathbb{Z}_{<0}$ is given in \cite[Page 7, case
{\bf a2}]{vPR}. The matrix $\Psi_{0}=\Psi_{0}^{m<0}$ is given by
$$\Psi_{0}^{m<0}(y)=\begin{pmatrix} \frac{m-y(m-i)}{i} & y^{\frac12} \\ 1 &
y^{\frac12} \end{pmatrix} = \begin{pmatrix} 0 & 1 \\ 1 & 0 \end{pmatrix} \,
\Psi_{0}^{(n,m-1,i+n)}(y) \, \begin{pmatrix} 0 & 1 \\ 1 & 0 \end{pmatrix},$$
where $\Psi_{0}^{(n,-m-1,n-i)}$ is given in \eqref{eq:Psi_0_Case_a2}. The case
{\bf a2} is therefore contained into the case {\bf a1}. 
\end{remark}

\subsection{Case {\bf c1} }

Now we consider the Gelfand pair
$(G,K)=(\mathrm{Sp}(2n),\mathrm{Sp}(2n-2)\times\mathrm{Sp}(2))$. This example
is studied in \cite[Page 7, case {\bf a1}]{vPR}. For $n\geq 3$ we have
$\alpha=2n-3$, $\beta=1$. We have the following expressions \begin{equation*}
\Psi_0(y)=\begin{pmatrix} \sqrt{y} & 1 \\ \sqrt{y} & \frac{y(n-1)-1}{n-2}
\end{pmatrix}, \qquad T=\begin{pmatrix} 1 & 0 \\ 0 & n-2 \end{pmatrix}.
\end{equation*} We also have  $a(y)=2-2yn$, \begin{equation*}
F(t)=\begin{pmatrix} {\frac {4\,{y}^{2}n+4\,yn-{y}^{2}-18\, y+3}{4y \left( y-1
\right) }}&-\frac{2\sqrt{y}}{y-1} \\ \noalign{\medskip}-{\frac {2\sqrt {y}
\left( n-2 \right) }{y-1}}& {\frac {2y}{y-1}} \end{pmatrix}, \end{equation*} and
the matrix-valued differential operator $D=y(1-y)\, \partial^2_y+\partial_y,
(C-yU)-V$, where \begin{equation*} C=\begin{pmatrix} \frac{2n-1}{n-1} &
\frac{n-2}{n-1} \\ \frac{1}{n-1} & \frac{3n-4}{n-1} \end{pmatrix}, \qquad
U=\begin{pmatrix} 2n+1 & 0 \\ -1 & 2n+2 \end{pmatrix},\qquad V=\begin{pmatrix} 0
& 0 \\ 0 & 2n-2 \end{pmatrix}. \end{equation*}
If we define $F^{(\kappa)}$ as in
Theorem \ref{thm:deformation}, a straightforward verification shows that the
unique solution $T^{(\kappa)}$ to \eqref{eq:conditionT} normalized by
$T^{(0)}=T$ is given by 
$$T^{(\kappa)}=\begin{pmatrix} 1 & 0 \\ 0 &
\frac{2(n-2)}{\kappa+2} \end{pmatrix}.$$ 
Now Theorem \ref{thm:deformation} says
that $(W^{(\kappa)},D^{(\kappa)})$ with 
\begin{align*}
W^{(\kappa)}(y)&=y^{\alpha+\kappa}(1-y)^{\beta+\kappa} \Psi_0(y) \, T^{(\kappa)}
\, \Psi_0(y)^*,\\ &= y^{(\kappa+1)}(1-y)^{(\kappa+2n-3)} \, \begin{pmatrix}
2\,{\frac {2\,y+y\kappa+2\,n-4}{2+\kappa}}&2\,{ \frac
{2\,yn-2+y\kappa}{2+\kappa}}\\ \noalign{\medskip}2\,{\frac {2\,yn-2+y
\kappa}{2+\kappa}}&2\,{\frac {2\,{y}^{2}{n}^{2}-4\,{y}^{2}n-2\,yn+\kappa\,yn+2
\,{y}^{2}-2\,y\kappa+2}{ \left( 2+\kappa \right)  \left( n-2 \right) }}
\end{pmatrix},\\
D^{(\kappa)}&=y(1-y)\partial_y^2+\partial_y(C+\kappa-y(U+2\kappa))-(V+\kappa
U+\kappa(\kappa-1)), 
\end{align*} 
is a MVCP. Moreover, if we compute explicitly
the functions $\Gamma_2^{(\kappa)}$ and $\Gamma_1^{(\kappa)}$ given in
\eqref{eq:def_Gamma2} and \eqref{eq:def_Gamma2} respectively, we obtain
\begin{align*} 
\Gamma_2^{(\kappa)}&= y^2 \, \begin{pmatrix} -1&- \left( 3+\kappa
\right) ^{-1} \\ \noalign{\medskip}0&-{\frac {2+\kappa}{3+\kappa}}
\end{pmatrix}+ y\, \begin{pmatrix} {\frac {\kappa\,n-1+2\,n-\kappa}{ \left(
3+\kappa \right)  \left( n-1 \right) }}&{\frac {1}{ \left( 3+\kappa \right)
\left( n-1 \right) }}\\ \noalign{\medskip}{\frac {n-2}{ \left( 3+\kappa \right) 
\left( n-1 \right) }}&{\frac {\kappa\,n+3\,n-4-\kappa}{ \left( 3+ \kappa \right)
 \left( n-1 \right) }} \end{pmatrix}, \\
\Gamma_1^{(\kappa)}&= y \, \begin{pmatrix} -2\,n-1-2\,\kappa&-{\frac
{1+\kappa}{3+\kappa}} \\ \noalign{\medskip}0&-2\,{\frac { \left( 2+\kappa
\right)  \left( \kappa+n +1 \right) }{3+\kappa}} \end{pmatrix} \\
&\hspace{5cm}+ \begin{pmatrix}
{\frac {-3-{\kappa}^{2}+6\,n-2\,\kappa+4\,\kappa\,n+ {\kappa}^{2}n}{ \left(
3+\kappa \right)  \left( n-1 \right) }}&{\frac {3+2\, \kappa}{ \left( 3+\kappa
\right)  \left( n-1 \right) }}\\ \noalign{\medskip} {\frac { \left( n-2 \right) 
\left( 2\,n+1+2\,\kappa \right) }{ \left( 3+ \kappa \right)  \left( n-1 \right)
}}&{\frac {{\kappa}^{2}n+7\,n+6\,\kappa\,n-8 \,\kappa-10-{\kappa}^{2}}{ \left(
3+\kappa \right)  \left( n-1 \right) }} \end{pmatrix}.
\end{align*} so that $\Gamma_2^{(\kappa)}$ is a polynomial of degree two and
$\Gamma_1^{(\kappa)}$ is a polynomial of degree one. Therefore, it follows from
Proposition \ref{eq:D^nu_Gamma2_Gamma1} that the sequence of derivatives
$(\partial_y P^{(\kappa)}_d)$ is orthogonal with respect to $W^{(\kappa+1)}$.

\subsection*{Acknowledgement.} The research of Pablo Rom\'an is supported by the
Radboud Excellence Fellowship. P. Rom\'an was partially supported by CONICET
grant PIP 112-200801-01533, FONCyT grant PICT 2014-3452 and by SECyT-UNC.

\appendix
\section{}
\renewcommand*{\thesection}{\Alph{section}}

Here we give a proof of the explicit expression for the function $\Psi_0$ in
terms of Krawtchouk polynomials. We will use that $\Psi_0$ is a solution to the
equation \eqref{eq:Psi0_in_y} and the fact that the matrix $S$ can be
diagonalized explicitly. Note that the columns of $K$ are precisely the
different eigenvectors of the matrix $S$ from (\ref{eq: S}). More precisely we have
$$S\cdot K=\mathrm{diag}(-\ell,-\ell+1,\ldots,\ell-1,\ell) \cdot K,$$ see \cite[Lemma
4.3]{KvPR2} for the proof. Observe that the columns of $\Psi_0$ are
solutions to the first order differential equation \begin{equation}
\label{eq:First_O_F} y(1-y)\,F'(y)+\frac12\left(S-\ell+2\ell y\right)F(y)=0,
\end{equation} \begin{proposition}\label{pro:solutions_equation} For
$j=0,1,\ldots,2\ell$, let $v_j$ be the eigenvector of $S$ with eigenvalue
$\ell-j$. Then \begin{equation*} \label{eq:solutions_F}
F_j(y)=y^{\frac{j}{2}}(1-y)^{\frac{2\ell-j}{2}}v_j, \end{equation*} is a
solution of \eqref{eq:First_O_F}. Moreover, $\{F_j\}_{j=0}^{2\ell}$ is a basis
of solutions of \eqref{eq:First_O_F}. \end{proposition} \begin{proof} This follows
by replacing $F_j$ in \eqref{eq:First_O_F}. The set $\{F_j\}_{j=0}^{2\ell}$ is
basis of solutions since $v_j$ are all linearly independent. \end{proof}

A simple calculation shows that \begin{equation} \label{eq:F_j_in_t}
F_j\left((1-\cos(t))/2\right)= 2^{-2\ell}  (-1)^\frac{j}{2}  e^{-i\ell
t}(e^{it}-1)^j(e^{it}+1)^{2\ell-j} v_j. \end{equation}

\begin{remark} \label{remark:leading_coef_Phi0} It follows directly from the
previous proposition that each entry of $F_j(\cos t)$ is a linear combination of
$\{ e^{-i\ell t},e^{-i(\ell-1) t},\ldots,e^{i\ell t}\}$. Moreover, for all
$j=0,1,\ldots,2\ell$, we have \begin{equation} \label{eq:leading_coef_F_j}
F_j((1-\cos(t))/2)=\left(2^{-2\ell}(-1)^{\frac{3j}{2}} e^{-i\ell
t}+\sum_{k=1}^{2\ell} a_k e^{i(-\ell+k) t} \right)v_j, \end{equation} for
certain coefficients $a_k$. \end{remark} \begin{lemma}
\label{lem:Phi_0_leading_term} Let $\Phi_0$ be given by \eqref{def:Phi_0}. Then
$$\Phi_0=M^{-1} \, e^{-i \ell t} + \sum_{k=1}^{2\ell} A_k e^{i(-\ell+k)
t },\qquad M=\sum_{j=0}^{2\ell} \binom{2\ell}{j} E_{j,j},$$ where $A_k$,
$k=1,\ldots, 2\ell$ are $(2\ell+1)\times(2\ell+1)$ matrices. \end{lemma}
\begin{proof} The proof follows directly from \eqref{def:Phi_0}. \end{proof}

\begin{proof}[Proof of Theorem \ref{prop:Psi_0_as_Krawtchouk}] Since $\Psi_0$ is
a solution of \eqref{eq:First_O_F} and $\{F_j\}$ is a basis of solutions, every
column of $\Psi_0$ is a linear combination of the functions $F_j$,
$j=0,\ldots,2\ell$. If we denote by $\Psi_0^{(k)}$ the $k$-th column of
$\Psi_0$, then there exist constants $a_0^{(k)},\ldots,a_{2\ell}^{(k)}$ such
that $$\Psi_0^{(k)}(y)=a_0^{(k)}F_0(y)+\cdots+a_{2\ell}^{(k)}F_{2\ell}(y).$$ If
we make the change of variables variable $\cos t=1-2y$, using that
$\Psi_0^{(k)}(\arccos(1-2y))=\Phi_0^{(k)}(t)$, it follows from
\eqref{eq:F_j_in_t} that \begin{equation} \label{eq:column_Phi_k_t}
\Phi^{(k)}_0(t)=2^{-2\ell}(-1)^{\frac{j}{2}}e^{-i\ell
t}[a_0^{(k)}(e^{it}+1)^{2\ell}v_0+\cdots + (-1)^{\ell}
a^{(k)}_{2\ell}(e^{it}-1)^{2\ell}v_{2\ell}] \end{equation} If we look at the
coefficient of $e^{-i\ell t}$ on both sides of \eqref{eq:column_Phi_k_t}, using
\eqref{eq:leading_coef_F_j} and Lemma \ref{lem:Phi_0_leading_term}, we obtain the
following equation, \begin{equation} \label{eq:leading_coef_ek}
\binom{2\ell}{k}^{-1} e_k = 2^{-2\ell} \sum_{j=0}^{2\ell} a_j^{(k)}
(-1)^{\frac{3j}{2}} v_j, \end{equation} where $e_k$ is the standard basis
vector. Let $\Gamma$ be the diagonal matrix given by
$\Gamma_{j,j}=(-1)^{\frac{3j}{2}}$. Then \eqref{eq:leading_coef_ek}  can be
written in matrix form as \begin{equation*} \binom{2\ell}{k}^{-1}  e_k =
2^{-2\ell} K \, \Gamma \, (a_0^{(k)},\ldots,a_{2\ell}^{(k)})^t. \end{equation*}
Then it follows that \begin{equation*} \label{eq:formula_for_ak}
(a_0^{(k)},\ldots,a_{2\ell}^{(k)})^t=2^{\ell}  \binom{2\ell}{2\ell-k}^{-1} 
\Gamma^{-1} K^{-1} e_k. \end{equation*}

In matrix form, \eqref{eq:column_Phi_k_t} can be written as 
\begin{equation}
\label{eq:Phi_0_ak} 
\widetilde \Phi_0^{(k)}(t)=K\widetilde
\Upsilon(t)M(a_0^{(k)},\ldots,a_{2\ell}^{(k)})^t, 
\end{equation} 
where
$\widetilde \Upsilon$ is the diagonal matrix 
\begin{equation*}
\widetilde\Upsilon(t)_{j,j}=2^{-\ell}(-1)^{\frac{j}{2}}\binom{2\ell}{i}e^{-i\ell
t}(e^{it}-1)^j(e^{it}+1)^{2\ell-j}. 
\end{equation*} If we replace this
expression in \eqref{eq:Phi_0_ak} we obtain 
\begin{equation*} \widetilde
\Phi_0^{(k)}(t)=2^{2\ell} \binom{2\ell}{2\ell-k}^{-1} K\widetilde
\Upsilon(t)\Gamma^{-1}K^{-1}e_k, 
\end{equation*} 
which leads to
\begin{equation}
\label{eq:column_Psi0_k}
\Psi_0^{(k)}(y)=2^{2\ell} K
\Upsilon(y)MK^{-1}M^{-1}e_k,
\end{equation}
The column vector \eqref{eq:column_Psi0_k} is precisely the $k$-th column of \eqref{eq:decomposition_Phi0}. This completes the proof of the theorem.
\end{proof}


We proceed to collect results and
formulas about Krawtchouk polynomials that we use in Section
\ref{sec:MVchebyshev}. We will need the following relations \begin{align}
\label{eq:ortho_K} \sum_{i=0}^{2\ell} \binom{2\ell}{i} \,
K_n(i)K_m(i)=\delta_{n,m}\, 2^{2\ell} \, \binom{2\ell}{n}^{-1},\\
\label{eq:three-term_K} -i\, K_n(i) = \frac{N-n}{2} \, K_{n+1}(i) - \frac{N}{2}
\, K_n(i) +\frac{n}{2} \, K_{n-1}(i).
\end{align}
Using the three-term recurrence relation \eqref{eq:three-term_K}  and
orthogonality \eqref{eq:ortho_K}, we obtain \begin{equation}
\label{eq:ortho_K_i} \sum_{i=0}^{2\ell} \binom{2\ell}{i} \, i\,
K_k(i)K_j(i)=2^{2\ell} \, \binom{2\ell}{k}^{-1} \begin{cases} -\frac{k+1}{2},&
j=k+1,\\ \ell,& j=k,\\ -\frac{2\ell-k+1}{2},& j=k-1,\\ 0, & \text{otherwise},
\end{cases} \end{equation} and \begin{equation} \label{eq:ortho_K_i2}
\sum_{i=0}^{2\ell} \binom{2\ell}{i} \, i^2\, K_k(i)K_j(i)=2^{2\ell} \,
\binom{2\ell}{k}^{-1}\begin{cases} \frac{(k+1)(k+2)}{4},& j=k+2,\\ -\ell(k+1),&
j=k+1,\\ (\ell(\ell+1/2)+k(\ell-k/2)),& j=k,\\ - \ell(2\ell-k+1),& j=k-1,\\
\frac{(2\ell-k+1)(2\ell-k+2)}{4},& j=k-2, \\ 0, & \text{otherwise} \end{cases}
\end{equation}

\begin{lemma}\label{lemma: rec kraw} The following recurrence relations hold:
\begin{eqnarray} \label{eq:diff_Krawt1}
(2\ell-i)K_{k}(i+1)&=&\frac{1}{2}(2\ell-k)K_{k+1}(i)+(\ell-k)K_{k}(i)-\frac{1}{2
}kK_{k-1}(i),\\ \label{eq:diff_Krawt2}
iK_{k}(i-1)&=&-\frac{1}{2}(2\ell-k)K_{k+1}(i)+(\ell-k)K_{k}(i)+\frac{1}{2}kK_{k-
1}(i) \end{eqnarray} \end{lemma}

\begin{proof} Note that we only have to prove one of the relations, because the
relations add up to the difference equation \cite[(1.10.5)]{KoekS}.
Alternatively, replace $i$ by $2\ell-i$ in the first equation and apply the
basic relation $\rFs{2}{1}{a,b}{c}{z}=(1-z)^{-a}\rFs{2}{1}{a,c-b}{c}{z/(z-1)}$
specialized in $z=2$. The basic relations \begin{multline*}
(2\ell-i)\rFs{2}{1}{-i-1,-k}{-2\ell}{z}+(i+k-2\ell)\rFs{2}{1}{-i,-k}{-2\ell}{z}=
\\ -k(z-1)\rFs{2}{1}{-i,-k+1}{-2\ell}{z}, \end{multline*} \begin{multline*}
(2\ell-k)\rFs{2}{1}{-i,-k-1}{-2\ell}{z}+(2k-2\ell+(i-k)z)\rFs{2}{1}{-i,-k}{-2
\ell}{z}=\\ -k(z-1)\rFs{2}{1}{-i,-k+1}{-2\ell}{z}, \end{multline*} imply the
result. Indeed, subtracting one half times the second from the first yields
\begin{multline*}
(2\ell-i)\rFs{2}{1}{-i-1,-k}{-2\ell}{z}=\frac{1}{2}(2\ell-k)\rFs{2}{1}{-i,-k-1}{
-2\ell}{z}+\\
(\ell-\frac{1}{2}zk+(\frac{1}{2}z-1)i)\rFs{2}{1}{-i,-k}{-2\ell}{z}-\frac{1}{2}
\rFs{2}{1}{-i,-k+1}{-2\ell}{z}, \end{multline*} which specializes to the desired
equation for $z=2$. \end{proof}

\end{document}